\documentclass{article}
\usepackage{amsmath,amsthm,amssymb,amscd}
\usepackage[all,cmtip,color]{xy}
\usepackage[german,frenchb,english]{babel}
\usepackage{amsmath}
\usepackage{amssymb}
\usepackage{hyperref}
\usepackage{mathrsfs} 
\usepackage{amssymb}
\usepackage{todonotes}
\usepackage{amssymb}
\usepackage[all,cmtip]{xy}
\input diagxy
\usepackage{hyperref}
\usepackage[margin = 2.5 cm]{geometry}
\usepackage[export]{adjustbox}
\usepackage{graphicx}

\DeclareMathOperator{\Tate}{\mathsf{Tate}}
\DeclareMathOperator{\Calk}{\mathsf{Calk}}
\DeclareMathOperator{\elTate}{\mathsf{Tate}^{\mathsf{el}}}
\DeclareMathOperator{\Proa}{\mathsf{Pro^a}}

\DeclareMathOperator{\Res}{\mathsf{Res}}

\DeclareMathOperator{\crepsilon}{{}^{\mathsf{cr}}\varepsilon}

\DeclareMathOperator{\Bin}{\mathsf{BinCh}}

\DeclareMathOperator{\bliss}{\mathsf{bliss}}

\DeclareMathOperator{\Ch}{Ch_{\geq 0}}

\DeclareMathOperator{\PSh}{PrSh}

\newcommand{\Cc}{\mathcal{C}}

\DeclareMathOperator{\colim}{\mathsf{colim}}
\renewcommand{\lim}{\mathsf{lim}}

\DeclareMathOperator{\coker}{coker}

\newcommand{\C}{\mathsf{C}}

\newcommand{\Spectra}{Sp}

\DeclareMathOperator{\dual}{\vee}

\DeclareMathOperator{\Pro}{\mathsf{Pro}}

\DeclareMathOperator{\Loc}{\mathsf{Loc}}

\DeclareMathOperator{\Ind}{\mathsf{Ind}}

\DeclareMathOperator{\G}{\mathbb{G}}

\DeclareMathOperator{\Sch}{\mathsf{Sch}}

\DeclareMathOperator{\red}{red}

\begin{document}
\newtheorem{definition}{Definition}[section]
\newtheorem{theorem}[definition]{Theorem}
\newtheorem{proposition}[definition]{Proposition}
\newtheorem{corollary}[definition]{Corollary}
\newtheorem{conj}[definition]{Conjecture}
\newtheorem{lemma}[definition]{Lemma}
\newtheorem{rmk}[definition]{Remark}
\newtheorem{cl}[definition]{Claim}
\newtheorem{example}[definition]{Example} 
\newtheorem{claim}[definition]{Claim}
\newtheorem{ass}[definition]{Assumption}
\newtheorem{warning}[definition]{Dangerous Bend}
\newtheorem{porism}[definition]{Porism}

\date{}
\author{Michael Groechenig}

\title{De Rham epsilon factors for flat connections on higher local fields\let\thefootnote\relax\footnotetext{This project has received funding from the European Union's Horizon 2020 research and innovation programme under the Marie Sk\l odowska-Curie Grant Agreement No. 701679. \\ \includegraphics[height = 1cm,right]{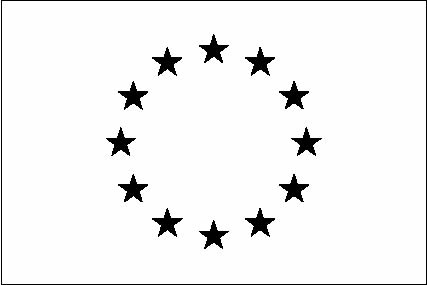}
} 
}
\maketitle 

\abstract{This note is a companion to the author's \emph{Higher de Rham epsilon factors}. Using Grayson's binary complexes and the formalism of $n$-Tate spaces we develop a formalism of graded epsilon lines, associated to flat connections on a higher local field of characteristic $0$. The definition is based on comparing a Higgs complex with a de Rham complex on the same underlying vector bundle.}

\tableofcontents
 
\section{Linear differential equations on higher local fields}

This section is devoted to a study of flat connections $(E,\nabla)$ on the higher local field $F_n = k((t_1))\cdots((t_n))$, where $k$ denotes as always a field of characteristic $0$. A lot of the material will be familiar to the expert. We include the often well-known proofs to demonstrate that the behaviour of differential equations on $F_n$ is not any different from the theory of $D$-modules on algebraic varieties. We do not lay claim to originality in this subsection.

\subsection{Definitions and basic properties}

\begin{definition}
Let $(E,\nabla)$ be a flat connection on $F_n$. The corresponding de Rham complex is the chain complex of $F_n$-vector spaces
$\Omega_{\nabla}^{\bullet} = [E \to^{\nabla} E \otimes \Omega_X^1 \to^{\nabla} \cdots \to^{\nabla} E \otimes \Omega_X^n].$ For $i \in \mathbb{Z}$ we denote the $i$-th cohomology group of this complex by $H^i_{\nabla}(E)$.
\end{definition}

\subsection{Cyclic vectors}\label{sub:cyclic}

The proof of the following lemma follows the lecture notes on algebraic $D$-modules by Braverman--Chmutova \cite[Lecture 3]{bc}. We denote by $D$ the ring of formal differential operators on $F_1 = k((t))$, which we define to be the free $k$-algebra $F_1\langle\partial\rangle$ modulo the relation $[\partial,t] =1$. It is clear that the centre is given by  $Z(D) = k$ (relying heavily on the assumption that $k$ has characteristic $0$).

\begin{lemma}[Cyclic Vector Lemma]\label{lemma:cyclic}
Let $(E,\nabla)$ be a flat connection on $F_1 = k((t))$. Then, as a $D$-module, there exists a cyclic vector $s \in E$, that is, a generator $Ds =E$.
\end{lemma}

\begin{proof}
The proof begins by observing that $D$ does not have any proper two-sided ideals, thus is a simple $k$-algebra. Indeed, let $I \subset D$ be a two-sided ideal, distinct from the trivial ideal $D$. It is clear that the intersection $F_1 \cap I$ must be $\{0\}$. We will use the tautological filtration on $D$ (also known as the order of differential operators) to establish that $I$ must be the zero ideal. If there is a non-zero element in $I$, then there is a minimal integer $r \geq 1$, such that $I$ contains an element $p$ of order $\leq r$. Since $p$ cannot be central in $D$ (as $Z(D)=k$) we obtain the existence of either a differential operator $q$ of order $\leq 1$, such that $[p,q] \neq 0$. The definition of two-sided ideals shows that $[p,q] \in I$. Hence it is a non-zero element of $I$ of order $\leq r-1$ which contradicts the minimality assumption on $r$.

Since $E$ is a finite-dimensional $F_1$-vector space, it has finite length as a $F_1$-module and therefore also as a $D$-module. We assume by induction that the existence of a cyclic vector has already been established for $D$-modules of smaller length than the one of $E$ (the case of length $1$ being automatically true). There exists a short exact sequence of $D$-modules
$$0 \to F \to E \to^{\phi} M \to 0,$$
such that $F$ has length $1$. Let $m \in M$ be a cyclic vector for $M$. We assume by contradiction that $\phi^{-1}(m)$ does not contain a cyclic vector for $E$. Then we have $Ds \cap F = \{0\}$. Let $p \in D$ be a differential operator, such that $pm = 0$, and $f \in F$ an arbitrary element of $F$. We have $\phi(s+f) = m$, and therefore also $D(s +f)\cap F=\{0\}$. Since $p(s+f) = pf \in D(s+f) \cap F$, we must have that $pf = 0$. Therefore, the ideal $\mathrm{Ann}(m)$ contains the two-sided ideal given by $\mathrm{Ann}(F)$. 

This leaves us with two options to consider, $\mathrm{Ann}(F)$ equal to $\{0\}$ or $D$. The first case is not possible as then $F$ would be isomorphic to $D$ and hence $\dim_{F_1} F = \infty$ which is impossible. We conclude that $\mathrm{Ann}(m) = D$ which implies $M=0$ and therefore simplicity of $E$.
\end{proof}

\begin{corollary}
Let $\nabla$ be a flat connection on the trivial rank $m$ bundle $F_1^{\oplus m}$. There exist elements $a_0,\dots,a_{m-1} \in F_1$, such that the $k$-vector space of $\nabla$-flat sections is isomorphic to the $k$-vector space of solutions to the ordinary differential equation $y^{(m)} + a_{m-1}y^{(m-1)} + \cdots + a_0 y = 0$.
\end{corollary}

\begin{proof}
Let $s$ be a cyclic vector. The $F_1$-dimension of $E = F_1^{\oplus m}$ is $m$, which implies linear dependence of the $(m+1)$-tuple $s,\partial s,\cdots \partial^m s$. Therefore we can write $\partial^m s = -\sum_{i=0}^{m-1}a_i \partial^is$, for certain elements $a_i \in F_1$. We conclude that the $D$-module $E$ is equivalent to the quotient $D/(\partial^m + \sum_{i=0}^{m-1} a_i \partial^i)$. 
\end{proof}

Furthermore the cyclic vector lemma for $F_1$ implies the existence of cyclic vectors for $F_n$.

\begin{corollary}\label{cor:cyclic}
Let $(E,\nabla)$ be a flat connection on $F_n$. Then $E$ is cyclic as a module over $D_n = D_{F_n}$.
\end{corollary}

\begin{proof}
The ring $D_n$ contains a subring $D' \subset D_n$ which is generated by $F_n = F_{n-1}((t_n))$ and $\partial_n = \frac{\partial}{\partial t_n}$. The ring $D'$ is canonically isomorphic to $D_{F_{n-1}((t_n))/F_{n-1}}$. As a $D'$-module, $E$ is cyclic by virtue of Lemma \ref{lemma:cyclic}. This implies that $E$ is cyclic as a $D$-module.
\end{proof}

\subsection{Formal de Rham cohomology is finite-dimensional}\label{sub:deRham}

\begin{lemma}\label{lemma:wronski}
An $m$-tuple of formal Laurent series $y_1,\dots y_m \in k((t))$ is $k$-linearly independent if and only if the Wronskian $W(y_1,\dots,y_m)=|(y_i^{(j-1)})_{1 \leq i,j\leq m}|$ is non-zero.
\end{lemma}

\begin{proof}
It is clear that linear dependence of the $m$-tuple $y_1,\dots,y_m$ implies vanishing of the Wronskian. We prove the converse by induction on $m$, anchored to the case $m = 1$ which is a tautology. Let us assume that the statement has been proven for $(m-1)$-tuples and that $y_1 \neq 0$. This allows us to divide every element of the $m$-tuple by $y_1$ (which changes the Wronskian by a factor of $y_1^{-m}$) and does not affect $k$-linear independence. Henceforth we assume without loss of generality $y_1 = 1$ and compute
$$0=W(1,y_2,\dots,y_m) = W(y_2',\dots,y_m').$$
By assumption this shows that the $(m-1)$-tuple $y_2',\dots, y_m'$ is $k$-linear independent. Let $\lambda_2,\dots,\lambda_m \in k$ be scalars bearing witness to this fact, i.e., $\sum_{i=2}^m \lambda_i y_i' = 0$. Integrating this equation we obtain the existence of a scalar $\lambda_1 \in k$, such that $\lambda_1 \cdot 1 + \sum_{i=2}^m\lambda_i y_i = 0$. Since we assumed $y_1 = 1$ this establishes a linear dependence and hence concludes the proof.
\end{proof}

\begin{corollary}\label{cor:wronski}
Let $a_0,\dots,a_{m-1} \in k((t))$. The $k$-vector space of solutions to the ordinary differential equation 
$$y^{(m)} + a_{m-1} y^{(m-1)} + \cdots + a_0 y = 0$$
is of dimension at most $m$. 
\end{corollary}

\begin{proof}
Let $y_1,\dots,y_{m+1}$ be solutions to the ordinary differential equation above. We claim that they are $k$-linearly dependent. As we have seen above this is equivalent to vanishing of the Wronkian $W(y_1,\dots,y_{m+1}) = |(y_i^{(j-1)})_{1 \leq i,j\leq m+1}|$. The $F_1$-linear relation $y_i^{(m)} + a_{m-1}y_i^{(m-1)} + \cdots + a_0y_i = 0$ (which holds for all $i$ by assumption) describes a $F_1$-linear relation between the rows of the matrix $(y_i^{(j-1)})_{1 \leq i,j\leq m+1}$. This implies vanishing of the Wronskian.
\end{proof}

\begin{proposition}\label{prop:finite_dimensional}
Let $(E,\nabla)$ be a flat connection on $F_1=k((t))$. The $k$-vector spaces $\ker(\nabla)$ and $\coker(\nabla)$ are finite-dimensional.
\end{proposition}

\begin{proof}
We begin by showing finite-dimensionality of the kernel. As we have seen in Subsection \ref{sub:cyclic}, there exists an $m$-th order linear differential equation $y^{(m)} + a_{m-1}y^{(m-1)} + \cdots 
+ a_0 y = 0$ whose solutions form a $k$-vector space isomorphic to $\ker(\nabla)$. Corollary \ref{cor:wronski} shows that this $k$-vector space is finite-dimensional.

Finite-dimensionality of the cokernel is shown by dualising. Since the $k$-vector spaces $E$ and $E \otimes \Omega_{F_1}^1$ are infinite-dimensional, it helps to endow them with the $t$-adic topology. This is the topology induced by the natural valuation on $F_1 = k((t))$ (which is discrete on $k$). With respect to this topoology, $E$ is a so-called linearly locally compact $k$-vector space in the sense of Lefschetz (also known as Tate $k$-vector spaces). The residue pairing yields an isomorphism of the topological dual $E^{\prime}$ with the Tate $k$-vector space $E^{\vee} \otimes \Omega_{F_1}^1$ (where $E^{\dual}$ refers to the $K_{1}$-linear dual). With respect to this isomorphism, the dual of the connection $\nabla$ can be seen to be the map induced by the dual connection. We conclude finite-dimensionality of the cokernel of $\nabla$ from finite-dimensionality of the kernel of the dual connection $\nabla^{\vee}$.
\end{proof}

\begin{corollary}\label{cor:finite_dimensional}
Let $(E,\nabla)$ be a flat connection on $F_n$. For all degrees $i \in \mathbb{Z}$ the de Rham cohomology groups $H^i_{\nabla}(E)$ are finite-dimensional $k$-vector spaces.
\end{corollary}

Before giving the proof it seems appropriate to summarise the underlying ideas. We argue by induction on $n$, the case $n=0$ being a tautology. Let us assume that the assertion has been verified for flat connections on $F_{n-1}$. There is a canonical inclusion of fields $F_{n-1} \hookrightarrow F_n$, whose geometric counterpart is projection along the $t_n$-coordinate. The de Rham cohomology groups $H^i_{\nabla}(E)$ can be computed from a variant of the Leray spectral sequence relative to this map. Since $F_n = F_{n-1}((t_n))$ we can apply the finiteness result of Proposition \ref{prop:finite_dimensional} for the $n=1$ case and the induction hypothesis to conclude the result.

\begin{proof}
Let $(E,\nabla)$ be a flat connection on $F_n = F_{n-1}((t_n))$. We denote by $\Omega_{F_n/F_{n-1}}^1$ the $F_n$-linear subspace of $\Omega_{F_n}^1$, spanned by the elements $dt_n$. By definition it is orthogonal to the subspace 
$$(\Omega_{F_{n-1}}^1)_{F_n} \hookrightarrow \Omega_{F_n}^1,$$
generated by the elements $dt_1,\dots,dt_{n-1}$. By taking exterior powers and tensor products we obtain $F_n$-linear subspaces $(\Omega_{F_{n-1}}^p)_{F_n} \subset \Omega_{F_n}^p$ and $(\Omega_{F_{n-1}}^p)_{F_n} \otimes_{F_n} \Omega_{F_n/F_{n-1}}^1 \subset \Omega_{F_n}^{p+1}$.

There are natural differential operators $\nabla'\colon E \to E \otimes (\Omega_{F_{n-1}}^1)_{F_n}$, $\nabla''\colon E \to E \otimes \Omega_{F_n/F_{n-1}}^1$. More generally we may form the double complex
\[
\xymatrix{
E \ar[r]^-{\nabla'} \ar[d]_{\nabla''} & E \otimes (\Omega_{F_{n-1}}^1)_{F_n} \ar[r]^-{\nabla'} \ar[d] _{\nabla''}& \cdots \ar[r]^-{\nabla'} & E \otimes (\Omega_{F_{n-1}}^{n-1})_{F_n} \ar[d]_{\nabla''} \\
E \otimes \Omega_{F_n/F_{n-1}}^1 \ar[r]^-{\nabla'} & E \otimes (\Omega_{F_{n-1}}^1)_{F_n} \otimes \Omega_{F_n/F_{n-1}}^1 \ar[r]^-{\nabla'} & \cdots \ar[r]^-{\nabla'} & E \otimes (\Omega_{F_{n-1}}^{n-1})_{F_n} \otimes \Omega_{F_n/F_{n-1}}^1.
}
\]
Note that $\nabla = \nabla' + \nabla''$ and $\nabla' \nabla'' = - \nabla''\nabla'$. We conclude that the de Rham complex $\Omega_{\nabla}^{\bullet}(E)$ is isomorphic to the totalisation of this double complex. The associated spectral sequence is
$$E_2^{p,q} = H^p_{\nabla'}(H^q_{\nabla''}(E)) \Rightarrow H^{p+q}_{\nabla}(E).$$
By Proposition \ref{prop:finite_dimensional}, $H^q_{\nabla''}(E)$ is finite-dimensional, since $F_n=F_{n-1}((t))$. The induction hypothesis implies finite-dimensionality of $H^p_{\nabla'}(H^q_{\nabla''}(E))$. Therefore, every object appearing on the second page of our spectral sequence is finite-dimensional over $k$. Convergence of the spectral sequence (see \cite[Tag 0132]{stacks-project}) implies that the formal de Rham cohomology groups $H^{i}_{\nabla}(E)$ are of finite dimension over $k$.
\end{proof}

\section{De Rham epsilon-factors: definition and basic properties}

Given a flat connection $(E,\nabla)$ on the higher local field $F_n = k((t_1))\cdots((t_n))$ we define a formalism of epsilon-factors. Our theory depends on an $F_n$-linearly independent $n$-tuple of $1$-forms $\underline{\nu}=(\nu_1,\dots,\nu_n) \in \Omega^1_{F_n}$ and produces a graded (or super) line $\varepsilon_{\underline{\nu}}(E,\nabla)$. This graded line is naturally obtained from a homotopy point of the $K$-theory spectrum $K(k)$.

Our treatment begins with a recapitulation of the case $n = 1$, which is due to Beilinson--Bloch--Esnault \cite{MR1988970}. We reformulate their definition in terms of Grayson's binary complexes \cite{grayson2012algebraic}. This perspective allows us a glimpse at the higher-dimensional generalisation. We then define epsilon-lines in arbitrary dimension $n$, by using Grayson's iterated binary complexes, finite-dimensionality of formal de Rham cohomology (Corollary \ref{cor:finite_dimensional}), and crucially, almost-commutativity of the ring of differential operators.

\subsection{Recapitulation on Tate objects: categorical preliminaries}

The notion of Tate objects in exact categories goes back to Beilinson's \cite{MR923134} and Kato's \cite{MR1804933}, and is inspired by Lefschetz's theory of linearly locally compact vector spaces. We refer the reader to \cite{MR2872533} and \cite{Braunling:2014ys} for a detailed introduction to the theory and an overview of the history of Tate objects.

To an idempotent complete exact category $\C$ (see \cite{MR2606234} for an account of the general theory of exact categories) one defines an exact category of \emph{admissible} Ind objects $\Ind^a\C$. It is the full subcategory of the category of Ind objects $\Ind \C$, whose objects can be represented as a formal colimit $\colim_K X_k$, where for $F_1 \leq k_2 \in K$ the resulting morphism $X_{F_1} \to X_{k_2}$ is an admissible monomorphism in the sense of exact categories. The dual construction yields $\Pro^a\C$, an exact category of admissible Pro objects. 

The exact category $\Ind^a \Pro^a \C$ contains an extension closed full subcategory $\elTate(\C)$ whose objects are referred to as \emph{elementary Tate objects}. By definition, every elementary Tate object $V$ belongs to an admissible short exact sequence
$$L \hookrightarrow V \twoheadrightarrow D,$$
where $L \in \Pro^a\C$ and $D \in \Ind^a\C$. One also says that $L \subset V$ is a \emph{lattice}. The exact category of Tate objects $\Tate(\C)$ is defined to be the idempotent closure of $\elTate(\C)$.

The categorical quotient $\Ind^a(\C)/\C$ was studied by Schlichting in \cite{MR2079996} as the \emph{suspension category} $\mathcal{S}\C$. We denote its idempotent closure by $\Calk(\C)$ and refer to it as the category of \emph{Calkin objects}. There is a natural functor $\Tate(\C) \to \Calk(\C)$ given by sending $V$ to $V/L$ where $L$ is a lattice in $V$. In fact, one has $\elTate(\C)/\Pro^a(\C) \simeq \Ind^a(\C)/\C$.

\begin{rmk}
The most important input category $\C$ in this paper is $P_f(R)$, that is, the exact category of finitely generated projective $R$-modules. Here we denote by $R$ a commutative ring (with unit). In this case we will simply write $\Tate(R)$, in reference to $\Tate(P_f(R))$.
\end{rmk}

\subsection{A reformulation of BBE's theory}

For the purpose of this subsection we let $(E,\nabla)$ denote a flat connection on $F=F_1 = k((t))$. Recall that $E$ is an $F$-vector space and that $\nabla\colon E \to E\otimes \Omega_F^1$ is a $k$-linear map, satisfying the Leibniz identity. Since $F$ has a natural structure of a Tate vector space over $k$, so does $E$, and $\nabla$ can be easily seen to be a morphism in this category. Furthermore this almost defines an isomorphism of linearly locally compact vector spaces (in a technical sense).

\begin{lemma}
A connection $\nabla$ as above induces an isomorphism in the localised category $\elTate(k)/\Proa(k)$.
\end{lemma}

\begin{proof}
Consider the diagram of abstract $k$-vector spaces
\begin{equation}\label{diag:nabla}
\xymatrix{
& E \ar@{->>}[rd] \ar[rr]^-{\nabla} & & E\otimes \Omega_F^1 \ar@{->>}[rd] & & \\
K \ar@{^(->}[ru] & & E/K \ar@{^(->}[ru] & &  C
}
\end{equation}
where the diagonals are short exact sequences. The kernel $K$ and cokernel $C$ are finite-dimensional by virtue of Proposition \ref{prop:finite_dimensional}. Furthermore, the $k$-vector spaces $E$ and $E \otimes \Omega_F^1$ are endowed with a linearly locally compact topology (that is, they are Tate vector spaces). We claim that the diagram \eqref{diag:nabla} is a diagram in category of Tate vector spaces. This amounts to showing that all arrows represent continuous maps of topological vector spaces. Continuity of $\nabla$ has already been asserted and can be checked using an explicit presentation in terms of coordinates. It remains to verify that the diagonal morphisms are continuous. 

The $k$-linear map $i\colon K \hookrightarrow E$ is certainly continuous when endowing $K$ with the discrete topology. Since $K$ is finite-dimensional, this is the only possibility within the category of linearly locally compact $k$-vector spaces. Furthermore we assert that $i$ has a continuous retraction $r$. To construct $r$ one chooses a lattice $L \subset E$, such that $i(K) \subset L$. By definition of the category of linearly locally compact vector spaces there exists a retraction $r'\colon E \to L$. Moreover we know that the inclusion $K \hookrightarrow L$ has a retraction $r''$, since the category of linearly compact $k$-vector spaces is contravariantly equivalent to the category of discrete $k$-vector spaces (which is a split abelian category). We define $r = r'' \circ r'$. This shows that $E/K$ is a direct summand of $E$ and hence establishes continuity of the morphism $E \to E/K$.

Using the observation of the proof of Proposition \ref{prop:finite_dimensional} that the $k$-linear topological dual of the map $E \to E \otimes \Omega_F^1$ is equivalent to the dual connection $\nabla^{\vee}$, we obtain continuity of the maps on the right hand side of diagram \eqref{diag:nabla}.

In the category $\elTate(k)/\Proa(k)$ we have that morphisms with finite-dimensional kernel or cokernel are isomorphisms. In particular we see that $E \to E/K$ and $E/K \to E \otimes \Omega_F^1$ are isomorphisms in this localisation. Commutativity of the triangle in \eqref{diag:nabla} implies the assertion.
\end{proof}

In the article \cite{grayson2012algebraic}, Grayson approaches higher algebraic $K$-groups using the notion of (iterated) binary complexes. A binary complex in an additive category $\Cc$ is a complex with two differentials, respectively a pair of complexes sharing the same objects. More precisely it is defined to be a collection of objects $(X^j)_{j \in \mathbb{Z}}$ and morphisms $d_{1,j}, d_{2,j}\colon X^j \to X^{j+1}$ for all $j \in \mathbb{Z}$, such that for $i =1,2$ and all $j\in \mathbb{Z}$ we have  $d_{i,j+1} \circ d_{i,j} = 0$.

\begin{definition}
Let $\Cc$ denote an additive category. We denote by $\Bin(\Cc)$ the category whose objects are binary complexes $((X^j;d_{1,j},d_{2,j})_{j \in \mathbb{Z}})$, and morphisms 
$$((X^j;d_{1,j},d_{2,j})_{j \in \mathbb{Z}}) \to ((Y^j;d_{1,j},d_{2,j})_{j \in \mathbb{Z}})$$
are given by tuples $(f^j)_{j \in \mathbb{Z}}$, such that $d_{i,j} \circ f^j = f^{j+1} \circ d_{i,j}$ for $i=1,2$ and $j \in \mathbb{Z}$.
\end{definition}

\begin{definition}
Let $(E,\nabla)$ be a flat connection on $F=k((t))$ and $\nu \in \Omega_F^1$. We define $\Omega_{\nabla,\nu}^{\bullet}(E)$ to be the length $2$ binary complex, given by 
\[
[\xymatrix{
E \ar@<2pt>[r]^-{\nabla} \ar@<-2pt>[r]_-{\nu} & E \otimes \Omega_X^1
}]
\]
in the exact category $\Calk(k)$. For fixed $\nu$ this defines an exact functor $\Omega_{\nabla,\nu}^{\bullet}\colon \Loc(F) \to \Bin\left(\Tate(k)\right)$.
\end{definition}

One says that a binary complex $(X^{\bullet};d_1,d_2)$ in an exact category $\Cc$ is exact (or synonymously acyclic), if both differentials define an exact complex in $\Cc$ in the sense of \cite[Definition 8.8]{MR2606234}. We denote the corresponding extension-closed subcategory by $\Bin_{ex}(\Cc)$. Grayson shows in \cite{grayson2012algebraic} that there exists a canonical morphism of connective spectra $K(\Bin_{ex}(\Cc)) \to \Omega K(\Cc)$. In the article he assumes that $\Cc$ admits a calculus of long exact sequences, but this assumption is not needed to define the aforementioned morphism but only in proving the main result \cite[Theorem 4.3]{grayson2012algebraic} that $K(\Ch^{ex}(\Cc)) \to K(\Bin_{ex}(\Cc)) \to \Omega K(\Cc)$ is a fibre sequence of connected spectra. 

\begin{lemma}
Assume that $\nu$ is a non-zero $1$-form on $F$. Then for every flat connection $(E,\nabla)$, the binary complex $\Omega_{\nabla,\nu}^{\bullet}(E)$ is exact in the Calkin category. In particular we have an exact functor $$\Omega_{\nabla,\nu}\colon \Loc(F) \to \Bin_{ex}(\Calk(k)).$$
\end{lemma}

Consequently we have for $\nu \in \Omega_{F_1}^{1}\setminus \{0\}$ a well-defined map of $K$-theory spectra
$$\varepsilon_{\nu}\colon K(\Loc_{F_1}) \to K(k),$$
which we call the spectral $\varepsilon$-factor.

\subsection{Higher local fields and closed $1$-forms}

Let us denote by $\Loc_{F_n}$ the exact category of pairs $(E,\nabla)$, where $E$ is a finite rank $F_n$-vector space, and $\nabla$ a formal flat connection on $E$. In this subsection we will define for a $F_n$-linearly independent tuple of closed $1$-forms $\underline{\nu}=(\nu_1,\dots,\nu_n)$ on $F_n$ a map of spectra 
$$\varepsilon_{\underline{\nu}}\colon K(\Loc_{F_n}) \to K(k)$$
which generalises the epsilon-factor formalism for the $1$-local field $k((t))$. One could argue that the assumption that the forms $\nu_1,\dots,\nu_n$ are closed is a feature of the higher-dimensional case, since it is automatically true for $n= 1$. However we believe that it should not be needed in order to have a well-defined $\varepsilon_{\underline{\nu}}$ (see the companion article \cite{companion}). 

\begin{definition}\label{defi:multideRham}
Let $(E,\nabla)$ be a de Rham local system on $F_n$ and $\underline{\nu} \in (\Omega_{F_n}^1)^n$ denote an $F_n$-linearly independent tuple of closed $1$-forms. We define a multi-complex (whose totalisation is the formal de Rham complex) which is supported on the cube $\{0,1\}^n$ which we identify with the power set $\mathcal{P}(\{1,\dots,n\})$. For $M \subset \{1,\dots,n\}$ we use the notation
$$\Omega_{\nabla}^{M,\underline{\nu}}(E) = E \otimes F_n\langle \nu_{i_1} \wedge \cdots \wedge \nu_{i_{q}} \rangle \subset E \otimes \Omega_{F_n}^{q},$$
where $M = \{i_1 < \dots < i_{q}\}$. Furthermore, for every inclusion $j\colon M \hookrightarrow N=M \cup \{i\}$, we let
$$\nabla_j\colon \Omega_{\nabla}^{M,\underline{\nu}}(E) \to \Omega_{\nabla}^{N,\underline{\nu}}(E)$$ 
be the component of the connection $\nabla \colon E \otimes\Omega_{F_n}^{q} \to E \otimes\Omega_{F_n}^{q+1}$ with respect to the direct summands $\Omega_{\nabla}^{M,\underline{\nu}}(E)$ and $\Omega_{\nabla}^{N,\underline{\nu}}(E)$.
\end{definition}

It is important to emphasise that without the assumption that $\underline{\nu}$ is an $n$-tuple of closed forms, this definition would not produce a multicomplex. Indeed anti-commutativity of the resulting squares is guaranteed by this assumption. 

The proof of the following lemma is based on an observation on $n$-Tate objects. There is a natural embedding 
$$e \colon \C \hookrightarrow \Tate(\C).$$
This yields $n$ distinct ways to embed $\Tate^{n-1}(\C)$ into $\Tate^n(\C)$ which we denote by $e_1,\dots,e_n$.

\begin{lemma}
The multicomplex of Definition \ref{defi:multideRham} is acyclic in the category $\Calk^n(k)$.
\end{lemma}

\begin{proof}
A direct computation involving writing out power series in the variables $t_1,\dots,t_n$ allows one to infer the following claim.
\begin{claim}
Let $V = \sum_{j=1}^i a_i \frac{\partial}{\partial t_j}$, where $a_j \in F_n$, and $a_i \neq 0$. Then the kernel and cokernel of the map
$$\nabla_{V}\colon E \to E$$
belongs to the full subcategory 
$$e_i(\Tate^{n-1})\hookrightarrow \Tate^n.$$
\end{claim}

Applying this lemma to the vector fields given by the dual basis of $\underline{\nu}$ we deduce a similar statement for the differentials of the multicomplex refining the de Rham complex of $E$. In particular, since $e_i$ factors through the kernel of $\Tate^n \to \Calk^n$, we deduce the lemma.
\end{proof}

\begin{definition}
For $(E,\nabla) \in \Loc_{F_n}$ and $\underline{\nu} \in (\Omega_{F_n}^1)^n$ a linearly independent $n$-tuple of closed $1$-forms, we define a binary multicomplex supported on $\{0,1\}^n$ in the sense of Grayson. It is constructed by adding extra differentials to the multicomplex $\Omega_{\nabla}^{\blacksquare}(E)$. For every inclusion $j\colon M \hookrightarrow N=M \cup \{i\}$, we let
$$\nu_i \colon \Omega_{\nabla}^{M,\underline{\nu}}(E) \to \Omega_{\nabla}^{N,\underline{\nu}}(E)$$ 
be the component of the morphism $\wedge \nu_i \colon E \otimes\Omega_{F_n}^{q} \to E \otimes\Omega_{F_n}^{q+1}$ with respect to the direct summands $\Omega_{\nabla}^{M,\underline{\nu}}(E)$ and $\Omega_{\nabla}^{N,\underline{\nu}}(E)$. This binary multi-complex will be denoted by $\mathcal{E}^{\blacksquare}_{\nabla,\underline{\nu}}(E)$.
\end{definition}

\begin{definition}
We have an exact functor $\mathcal{E}_{\underline{\nu}}\colon \Loc_{F_n} \to (\Bin_{ex})^n(\Calk^n(k))$, sending $(E,\nabla)$ to the binary multicomplex $\mathcal{E}^{\blacksquare}_{\nabla,\underline{\nu}}(E)$.
\end{definition}

\begin{definition}
We define the map of spectra $\varepsilon_{\underline{\nu}}\colon K(\Loc_n) \to K(k)$ as the composition
$$K(\Loc_n) \to^{K(\mathcal{E}_{\underline{\nu}})} K((\Bin_{ex})^n(\Calk^n(k))) \to K(k),$$
where the second map combines Grayson's $K((\Bin_{ex})^n(\Calk^n(k))) \to \Omega^nK(\Calk^n(k))$ with Saito's equivalence $K(\Tate^n(k)) \simeq K(k)$ (see \cite{Saito:2012fk}).
\end{definition}

\subsection{Induction}

Let $F'_n/F_n$ be a finite extension of $n$-fields with last residue fields $k'/k$, and $\underline{\nu} \in (\Omega_{F_n}^1)^n$ be an $F_n$-linearly independent $n$-tuple. Observe that we have $\Omega_{F_n}^1 \subset \Omega_{F'_n}^n$. A de Rham local system $(E,\nabla)$ on $F'_n$ gives rise to a de Rham local system on $F_n$. The resulting exact functor is denoted by $\Ind\colon \Loc_{F'_n} \to \Loc_{F_n}$ and will be referred to as induction with respect to $F'_n/F_n$.

\begin{proposition}
There is a commutative diagram of spectra
\begin{equation}\label{eqn:induction}
\xymatrix{
K(\Loc_{F'_n}) \ar[r]^{\Ind} \ar[d]_{\varepsilon_{\underline{\nu}}} & K(\Loc_{F_n}) \ar[d]^{\varepsilon_{\underline{\nu}}} \\
K(k') \ar[r] & K(k).
}
\end{equation}
\end{proposition}

\begin{proof}
Note that every finite $F'_n$-vector space $E$ gives rise to an $n$-Tate $k'$-vector space. However, since $k'/k$ is a finite field extension, we can also view $E$ as an $n$-Tate $k$-vector space. Furthermore, the diagram of exact categories 
\[
\xymatrix{
P_f(F'_n) \ar[r] \ar[d] & P_f(F_n) \ar[d] \\
\Tate^n(k') \ar[r] & \Tate^n(k)
}
\]
commutes strictly.
We obtain the commutative diagram above by applying the $K$-theory functor $K(-)$ to the following strictly commutative diagram of exact categories
\[
\xymatrix{
\Loc_{F'_n} \ar[r]^{\Ind} \ar[d]_{\mathcal{E}_{\underline{\nu}}} & \Loc_{F_n} \ar[d]^{\mathcal{E}_{\underline{\nu}}} \\
\Bin_{ex}^n\Calk^n(k') \ar[r] & \Bin_{ex}^n\Calk^n(k)
}
\]
and using the natural map $K(\Bin_{ex}^n\Calk^n(k)) \to K(k)$.
\end{proof}

\subsection{Duality}

In this subsection we will prove the proposition below. We refer the reader to \cite{bghw} for the definition of duality for higher Tate objects.

\begin{proposition}\label{prop:duality}
There is a commutative diagram of spectra
\begin{equation}\label{eqn:dual}
\xymatrix{
K(\Loc_{F_n}) \ar[r]^{(-)^{\vee}} \ar[d]_{\varepsilon_{\underline{\nu}}} & K(\Loc_{F_n}) \ar[d]^{\varepsilon_{-\underline{\nu}}} \\
K(k) \ar[r]^{(-)^{\vee}} & K(k).
}
\end{equation}
\end{proposition}

Let $E$ be a finite-dimensional $F_n$-vector space. We denote by $E^{\vee}$ its $F_n$-linear dual. Since $E$ can be seen as an $n$-Tate $k$-vector space, it is also possible to consider the $k$-linear dual in the sense of $n$-Tate spaces, which we denote by $E'$. The following lemma relates the $F_n$-linear and $n$-Tate dual in a canonical fashion. It is the analogue of Serre duality for $n$-fields. Its proof is a straight-forward extension of the well-known case where $n=1$. 

\begin{lemma}\label{lemma:serre}
For every finite-dimensional $F_n$-vector space $E$ there is a natural isomorphism of $n$-Tate $k$-vector spaces 
$E \simeq E^{\vee} \otimes \Omega_{F_n}^n$ which is induced by the higher residue pairing $\Res (-,-)\colon E \times (E^{\vee} \otimes \Omega_{F_n}^n) \to k$.
\end{lemma}

For $p > 0$ we define $\Omega_{F_n}^{-p} = \Omega^{n-p}_{F_n} \otimes (\Omega_{F_n}^n)^{\dual}$.
It is easy to see that the $F_n$-linear dual of $\Omega^p_{F_n}$ is canonically isomorphic to $\Omega_{F_n}^{-p}$, by virtue of the twisted pairing $\wedge\colon \Omega_{F_n}^p \otimes \Omega_{F_n}^{n-p} \to \Omega_{F_n}^n$. 
With respect to this notation we obtain the following isomorphisms from Lemma \ref{lemma:serre} as an immediate consequence.

\begin{corollary}\label{cor:dual}
We have natural isomorphisms of $n$-Tate vector spaces $(E \otimes \Omega_{F_n}^p)' \simeq E^{\dual} \otimes \Omega_{F_n}^{n-p}$. 
\end{corollary}

With this isomorphism at hand we ask ourselves what the $n$-Tate dual of the map $E \otimes \Omega^p_{F_n} \to^{\nabla} E \otimes \Omega^{p+1}_{F_n}$ is. The answer to this question is given by the next corollary. We denote by $\nabla^{\vee}$ the dual connection on $E^{\vee}$.

\begin{corollary}\label{cor:dual2}
With respect to the isomorphism of Corollary \ref{cor:dual} we have $$\nabla' = -\nabla^{\vee}\colon E^{\vee} \otimes \Omega^{n-p-1}_{F_n} \to E^{\vee} \otimes \Omega^{n-p}_{F_n}.$$
\end{corollary}

\begin{proof}
For $s \in E$ and $t \in E^{\vee}$ we have the relation
$(\nabla s, t) + (s,\nabla^{\vee} t) = d(s,t)$. Applying the higher residue to both sides, we obtain $\Res(\nabla s, t) + \Res(s,\nabla^{\vee} t) = 0$, since the right hand side is exact. This shows $\nabla' t = -\nabla^{\vee} t$ for all $t \in E^{\vee}$.
\end{proof}

For a complex $A^{\bullet}$ in an exact category $\Cc$ we denote by $\iota^*A^{\bullet}$ the complex obtained by replacing all differentials $d_i\colon A^i \to A^{i+1}$ by $-d^i$. We can construct an isomorphism $A^{\bullet} \simeq \iota^*A^{\bullet}$ in terms of a commutative ladder.
\[
\xymatrix{
\cdots \ar[r]^{d_i} & A^i \ar[r]^{d^i} \ar[d]_{(-1)^i} & A^{i+1} \ar[d]^{(-1)^{i+1}} \ar[r]^{d^{i+1}} & \cdots \\
\cdots \ar[r]^{-d_i} & A^i \ar[r]^{-d^i} & A^{i+1} \ar[r]^{-d_{i+1}} & \cdots
}
\]
However there is a second choice of an isomorphism, which replaces all vertical arrows $(-1)^{i}$ by $(-1)^{i+1}$. So we see that there is a $\mu_2$-torsor of natural isomorphisms $A^{\bullet} \simeq \iota^*A^{\bullet}$. The same remark applies to binary complexes.
As a consequence one sees that the proof of the proposition below produces not just one commutative square as in \eqref{eqn:dual}, but two (we remind the reader that in the context of $\infty$-categories, commutativity of diagrams is an extra structure and not a property). 

\begin{proof}[Proof of Proposition \ref{prop:duality}]
Corollary \ref{cor:dual2} implies that the $n$-Tate dual of the binary multicomplex $\mathcal{E}^{\blacksquare}_{\nabla,\underline{\nu}}(E)$ is given by $\iota^*\mathcal{E}^{\blacksquare}_{\nabla,-\underline{\nu}}(E^{\vee})$.
Using one the natural isomorphisms above we obtain a of commutative diagram of exact categories 
\[
\xymatrix{
\Loc_{F_n} \ar[r]^{(-)^{\vee}} \ar[d]_{\mathcal{E}_{\underline{\nu}}} & \Loc_{F_n} \ar[d]^{\mathcal{E}_{-\underline{\nu}}} \\
\Bin_{ex}^n\Calk^n(k) \ar[r]^{(-)^{\vee}} & \Bin_{ex}^n\Calk^n(k).
}
\]
Using the natural map of spectra $K(\Bin_{ex}^n\Calk^n(k)) \to K(k)$ we conclude the proof of the proposition.
\end{proof}

\section{The variation of epsilon-factors in families}

\subsection{Epsilon-factors for epsilon-nice families}

As before we let $k$ be a field of characteristic $0$. We denote by $R$ a commutative $k$-algebra, and use $F_{n}^R$ as shorthand for the commutative ring $R((t_1))\cdots((t_n))$. As an $R$-module it carries the structure of an $n$-Tate $R$-module, and hence every finitely generated projective $F_n^R$-module admits a natural realisation in the exact category $\Tate^n(R)$. Given $E \in P_f(F_n^R)$, and a relative connection $\nabla\colon E \to E \otimes_{F_n^R} \Omega_{F_n^R/R}$, we investigate when it is possible to define an epsilon-factor $\varepsilon_{\underline{\nu}}(E,\nabla)$ in the $K$-theory spectrum $K(R)$.

\begin{definition}
Let $(E,\nabla) \in \Loc_R(F_n)$ be a de Rham local system, such that the formal de Rham multi-complex $\Omega^{\blacksquare,R}_{\nabla}(E)$ is an acyclic multi-complex in $\Calk^n(R)$. We say that $(E,\nabla)$ is a epsilon-nice (or blissful) $R$-family of flat connections. The corresponding full subcategory of $\Loc^R(F_n)$ will be denoted by $\Loc^R_{\bliss}(F_n)$.
\end{definition}

Since formation of the de Rham complex $\Omega_{\nabla}^R(E)$ is an exact functor, and acyclic complexes form a fully exact and idempotent complete subcategory, we see that  $\Loc^R_{\bliss}(F_n)$ is an exact and idempotent complete category.

\begin{rmk}
\begin{itemize}
\item[(a)] In \cite{MR1988970} such $R$-families of flat connections are called nice.

\item[(b)] Intuitively speaking the irregularity type of the connection stays constant in epsilon-nice $R$-families. However, just as in \emph{loc. cit.} we prefer to define epsilon-nice families in the abstract manner above.
\end{itemize}
\end{rmk}

For epsilon-nice families one can define a (spectral) de Rham epsilon factor, for every $R$-linearly tuple of closed relative $1$-forms $(\nu_1,\dots,\nu_n) \in \Omega^1_{F_n} \widehat{\otimes}_k R$. 

\begin{definition}
We denote by $\mathcal{E}^{\blacksquare}_{\underline{\nu}}$ the exact functor $\Loc^R_{\bliss}(F_n) \to \Bin^n_{ex}\Calk^n(R)$. Applying the $K$-theory functor we obtain a morphism of spectra $\varepsilon_{\underline{\nu}}^R\colon K(\Loc^R_{\bliss}(F_n)) \to K(R)$.
\end{definition}

\subsection{The epsilon-crystal}

Epsilon-factors in families are naturally endowed with a partial connection (meaning that we cannot covariantly derive along all vector fields, but only along a fields belonging to a subbundle). Since we are working with homotopy points of $K$-theory spectra this is best formalised using a crystalline approach. Before stating the main result of this subsection we have to introduce some notation.

\begin{definition}
\begin{itemize}
\item[(a)] For a scheme $S$ we denote by $S^{dR}$ the presheaf on the category of schemes given by $T \mapsto S(T^{\red})$. We have a natural transformation $S \to S^{\red}$.

\item[(b)] A functor $H \colon \Sch \to \Spectra$ extends to a functor $H\colon \PSh(\Sch_{Nis}) \to \Spectra$, such that for a sheaf $F$, we have $$H(F) \simeq \lim_{T / F} H(T).$$
In particular we get a well-defined spectrum $H(S^{dR})$ for every scheme $S$. We call it the spectrum of $F$-crystals over $S$.
 
\item[(c)] We refer to the spectrum $K(S^{dR})$ as the spectrum of $K$-crystals of $S$. Homotopy points thereof will be referred to as $K$-crystals defined over $S$.
\end{itemize}
\end{definition}

Using the natural map of Nisnevich sheaves $K \to B^{\mathbb{Z}}\G_m$ we see that every $K$-crystal on $S$ gives rise to a (graded) line with a flat connection on $S$. We will now show that the formalism of $\varepsilon$-factors has a natural crystalline refinement. This is the higher rank generalisation of the epsilon connection defined in in section 3 of \cite{MR1988970} (furthermore they did not consider $K$-crystals). The proof is omitted, as it is based on the $(\mathbb{P}^1,\infty)$-invariance property of algebraic $K$-theory detailed in the companion paper \cite{companion} to construct the epsilon connection.

\begin{proposition}
Let $(E,\nabla)$ be an object of $\Loc(F_n)$, $S$ an affine $k$-scheme and $\underline{\nu}=(\nu_1,\dots,\nu_n)$ a basis of $\Omega^1_{F_n^S/S}$. 
There exists a map $\crepsilon_{\underline{\nu}}\colon K(\Loc(F_n)) \to K(S^{dR})$, which renders the diagram
\[
\xymatrix{
K(\Loc(F_n)) \ar[r]^-{\crepsilon_{\underline{\nu}}} \ar[rd]_{\varepsilon_{\underline{\nu}}} & K(S^{dR}) \ar[d] \\
& K(S)
}
\]
commutative.
\end{proposition}

\subsection*{Concluding remarks}

In \cite{companion} the author introduced a theory of de Rham epsilon factors for $D$-modules on higher-dimensional varieties which is supported on points and gives rise to an epsilon-crystal. The definition in \emph{loc. cit.} is a continuation of Patel's theory of de Rham epsilon factors introduced in \cite{MR2981817}. We expect these two notions of de Rham epsilon factors to be equivalent.

\subsection*{Acknowledgements}

The author thanks Oliver Braunling, H\'el\`ene Esnault, Javier Fres\'an, Markus Roeser and Jesse Wolfson for many interesting conversations about epsilon factors, differential equations and Tate objects.

\bibliographystyle{amsalpha}
\bibliography{epsilon.bib}

\bigskip
\noindent E-mail: \url{m.groechenig@fu-berlin.de}\\
Address: FU Berlin, Arnimallee 3, 14195 Berlin

\end{document}